\documentclass[letterpaper, 10 pt, conference]{ieeeconf}

\IEEEoverridecommandlockouts 
\makeatletter
\newcommand{\removelatexerror}{\let\@latex@error\@gobble}
\makeatother
\usepackage[T1]{fontenc}
\usepackage[utf8]{inputenc}

\usepackage{cite}
\usepackage{dsfont}
\usepackage[super]{nth}

\newcommand{\xvbox}[2]{\makebox[#1][l]{#2}} 

\newcommand{\Pis}[1]{\Pi_{\mathrm{state}}(#1)}
\newcommand{\Pit}[1]{\Pi_{\mathrm{traj}}(#1)}
\newcommand{\Pic}[1]{\Pi_{\mathrm{cost}}(#1)}

\usepackage[english]{babel}
\usepackage[usenames,dvipsnames,table]{xcolor}
\usepackage{amsmath,amsthm,amssymb}
\usepackage{graphicx}
\usepackage{url}
\usepackage{float}

\usepackage{tikz}
\usetikzlibrary{shapes,arrows,positioning,automata}

\usepackage[mathscr]{euscript}
\usepackage{mathtools}

\usepackage{xargs} 
\usepackage{subcaption} 
\usepackage[font=footnotesize,labelfont=footnotesize]{caption}
\usepackage[noend,ruled,shortend,commentsnumbered,linesnumbered]{algorithm2e}
\usepackage{etoolbox} \let\bbordermatrix\bordermatrix
\patchcmd{\bbordermatrix}{8.75}{4.75}{}{}

\renewcommand{\baselinestretch}{.98}

\newcommand{\real}{\mathbb{R}}
\newcommand{\realnonnegative}{{\mathbb{R}}_{\ge 0}}

\newcommand{\naturalnumbers}{\mathbb{N}}
\newcommand{\norm}[1]{\ensuremath{\| #1 \|}}
\newcommand{\until}[1]{[\, #1 \,]}
\newcommand{\map}[3]{#1:#2 \rightarrow #3}
\newcommand{\setdef}[2]{\{#1 \; | \; #2\}}

\newcommand{\setmap}[3]{#1:#2 \rightrightarrows #3}
\newcommand{\ones}{\mathbf{1}}

\newcommand{\myemphc}[1]{\emph{#1}} 
\newcommand{\setr}[1]{\{#1\}}

\newcommand{\abs}[1]{|#1|}

\newcommand{\xtraj}{\mathsf{x}}
\newcommand{\utraj}{\mathsf{u}}
\newcommand{\Unsafe}{\mathcal{U}\mathcal{I}}
\newcommand{\JJ}{\mathcal{J}}
\newcommand{\Qo}{\overline{Q}}
\newcommand{\DDo}{\overline{\DD}}
\newcommand{\drmpc}{\mathtt{DR\_MPC}}

\newcommand{\st}{\operatorname{subject \text{$\, \,$} to}}
\renewcommand{\st}{\operatorname{s.t.}}

\newcommand{\Eb}{\mathbb{E}}
\newcommand{\Pb}{\mathbb{P}}

\newcommand{\Db}{\mathbb{D}}
\newcommand{\Data}{\widehat{\mathcal{W}}}

\newcommand{\DD}{\mathcal{D}}

\newcommand{\II}{\mathcal{I}}

\newcommand{\NN}{\mathcal{N}}
\newcommand{\OO}{\mathcal{O}}

\newcommand{\PP}{\mathcal{P}}

\newcommand{\UU}{\mathcal{U}}

\newcommand{\WW}{\mathcal{W}}
\newcommand{\XX}{\mathcal{X}}
\newcommand{\YY}{\mathcal{Y}}

\newcommand{\CVaR}{\operatorname{CVaR}}
\newcommand{\VaR}{\operatorname{VaR}}

\newcommand{\eps}{\epsilon}

\newcommand{\what}{\widehat{w}}

\newcommand{\costgo}[2]{J_{(#1:#2)}}

\newcommand{\jth}{j^{\mathrm{th}}}
\newcommand{\oprocendsymbol}{\hbox{$\bullet$}}
\newcommand{\oprocend}{\relax\ifmmode\else\unskip\hfill\fi\oprocendsymbol}
\newcommand{\longthmtitle}[1]{\mbox{}\textup{\textsl{(#1):}}}

\allowdisplaybreaks

\newcommand{\ifinclude}[1]{}

\renewcommand{\theenumi}{(\roman{enumi})}

\makeatletter
\newcommand{\thickhline}{
  \noalign {\ifnum 0=`}\fi \hrule height 1pt
  \futurelet \reserved@a \@xhline
}
\newcolumntype{"}{@{\hskip\tabcolsep\vrule width 1pt\hskip\tabcolsep}}
\makeatother

\newtheorem{theorem}{Theorem}[section]
\newtheorem{proposition}{Proposition}[section]

\theoremstyle{definition}

\newtheorem{assumption}{Assumption}[section]
\newtheorem{remark}[theorem]{Remark}

\usepackage[normalem]{ulem} 
  
\definecolor{new}{rgb}{0.55,0,0.55}

\usepackage[prependcaption,colorinlistoftodos]{todonotes}

\title{Data-driven distributionally robust iterative risk-constrained \\ model predictive control} 
\author{Alireza Zolanvari \qquad Ashish Cherukuri \thanks{The authors are with the Engineering and Technology Institute Groningen, University of Groningen. Email: \texttt{\{a.zolanvari,a.k.cherukuri\}@rug.nl}. This work was partly supported with a scholarship from the Data Science and Systems Complexity (DSSC) Center, University of Groningen. }}

\begin{document}
\maketitle
\thispagestyle{empty}
\pagestyle{empty}

\begin{abstract}
	This paper considers a risk-constrained infinite-horizon optimal control problem and proposes to solve it in an iterative manner.
	Each iteration of the algorithm generates a trajectory from the starting point to the target equilibrium state by implementing a distributionally robust risk-constrained model predictive control (MPC) scheme. At each iteration, a set of safe states (that satisfy the risk-constraint with high probability) and a certain number of independent and identically distributed samples of the uncertainty governing the risk constraint are available. These states and samples are accumulated in previous iterations. The safe states are used as terminal constraint in the MPC scheme and samples are used to construct a set of distributions, termed ambiguity set, such that it contains the underlying distribution of the uncertainty with high probability. The risk-constraint in each iteration is required to hold for all distributions in the ambiguity set. We establish that the trajectories generated by our iterative procedure are feasible, safe, and converge asymptotically to the equilibrium. Simulation example illustrates our results for the case of finding a risk-constrained path for a mobile robot in the presence of an uncertain obstacle. 
\end{abstract}

\section{Introduction}\label{sec:intro}

Practical control systems often operate in uncertain environments, for example, a mobile robot navigating in the presence of obstacles. Safe optimal control in such situations can be modeled in many different ways. On the one hand, robust approaches consider the worst-case effect of the uncertainty on control design. On the other hand, popular probabilistic approaches model safety as chance-constraints in the optimal control problem and design deterministic or sample-based algorithms to solve it. A convenient strategy to balance these approaches is to consider appropriate risk constraints. We adopt this approach in our work and define a risk-constrained infinite-horizon optimal control problem. We assume that the task needs to be performed in an iterative way and the data regarding the uncertainty is incrementally revealed as iterations progress. We design an iterative method that combines the notions of learning model predictive control~\cite{UR-FB:17-tac} and distributionally robust risk constraints~\cite{ARH-AC-JL:19-acc}.

\subsubsection*{Literature review}
Optimization problems with worst-case expectation over a set of distributions, either in objective or constraints, is commonly termed as distributionally robust (DR) optimization~\cite{HR-SM:19-arXiv}. The considered set of distributions is referred to as the ambiguity set. The DR framework is particularly attractive when the data regarding uncertainty is less. In this case, the decision-maker can construct the ambiguity set of appropriate size to tune the out-of-sample performance. Thus, DR optimization lends itself as a fitting tool for ensuring safety in uncertain systems. With this motivation, works~\cite{CM-SL:21,JC-JL-FD:21,PC-PP:21,MS-PP:21-arXiv} explore distributional robustness in model predictive control (MPC). Further, recent works~\cite{PS-DH-AB-PP:18,SS-YC-AM-MP:19-tac} investigate risk-averse MPC for Markovian switched systems, while making use of the connection that the so-called coherent risk measure of a random variable is equivalent to the worst-case expectation over a set of distributions.

While most of the above-listed works on MPC consider stochastic systems, we only focus on uncertain environments. This setup finds application in risk-averse motion planning, where our work is related to~\cite{AH-GCK-IY:19,AD-MA-JWB:20}. Here risk constraints encode safety against collision. When only few samples regarding the uncertainty are available,~\cite{AH-IY:20-ICRA,AH-IY:21-arXiv} use distributional robustness to ensure safety. However, none of these works explore the possibility of executing the task in an iterative manner. Such a method is appealing when data regarding the uncertainty is scarce at the beginning and more samples get revealed as the task is done repeatedly. As a consequence, the environment can be explored progressively. The safety ensured in this process can be tuned using DR constraints. To actualize such a method, we make use of the learning model predictive framework introduced in~\cite{UR-FB:17-tac}. Here, at each iteration, a part of the state space is explored and stored for future iterations where these states are used as terminal constraints. In \cite{MB-CV-FB:20}, a learning-based MPC has been developed to tackle the uncertainties in the problem's constraints in a safe procedure. However, these strategies aim at satisfying robust and not risk constraints.

\subsubsection*{Setup and contributions} 
We define an infinite-horizon optimal control problem for a discrete-time deterministic system, where the state is subjected to a conditional value-at-risk constraint. The goal is to take the state from a starting point to a target equilibrium. Our main contribution is the design of the distributionally robust iterative MPC scheme that progressively approximates the solution of the infinite-horizon problem. In our procedure, at each iteration, we generate a trajectory using an MPC scheme, where a DR constrained finite-horizon problem is solved repeatedly. We assume a general class of ambiguity sets that are defined using the data collected in previous iterations. The terminal constraint in the finite-horizon problem enforces the state to lie in a subset of the safe states sampled in previous iterations. Once a trajectory is generated, the samples of the uncertainty collected in the iteration are added to the dataset and the sampled safe set is updated appropriately. 

We establish three properties for our method. Under the assumption that a robustly feasible trajectory is available at the first iteration, we show that each iteration is recursively feasible and safe, where safety means satisfying the risk-constraint with high probability. We prove that each trajectory asymptotically converges to the target state. Lastly, we give conditions under which the set of safe states grows and the cost of trajectory decreases as iterations progress. We apply our algorithm for finding a risk-averse path for a mobile robot in the presence of an uncertain obstacle. 

The attractive aspect of our iterative algorithm is the fact that safety and cost-performance can be tuned using distributional robustness, irrespective of the number of available samples. This feature is quite useful in initial iterations where only few samples of the uncertainty might be available.

\section{Preliminaries}\label{sec:prelims}
Here we collect notation and mathematical background.  
\subsubsection{Notation}\label{subsec:notation}
Let $\real$, $\realnonnegative$, and $\naturalnumbers$ denote the set of real, non-negative real, and natural numbers, resp. The set of natural numbers excluding zero is denoted as $\naturalnumbers_{\ge 1}$. Let $\norm{\cdot}$ and $\norm{\cdot}_1$ denote the Euclidean $2$- and $1$-norm, resp. For $N \in \naturalnumbers$, we denote $[N] := \{0,1,\dots,N\}$. Given $x \in \real$, we let $[x]_+ = \max(x,0)$. Given two sets $X$ and $Y$, a set-valued map $\setmap{f}{X}{Y}$
associates to each point in $X$ a subset of $Y$. The $n$-fold Cartesian product of a set $\mathcal{S}$ is denoted as $\mathcal{S}^n$. 
The $n$-dimensional unit simplex is denoted as $\Delta_{n}$.

\subsubsection{Conditional Value-at-Risk}\label{subsec:cvar}
We review notions on conditional value-at-risk (CVaR) from~\cite{AS-DD-AR:14}. Given a real-valued random variable $Z$ with probability distribution $\Pb$ and $\beta \in (0,1)$, the \myemphc{value-at-risk} of $Z$ at level $\beta$, denoted $\VaR_\beta^{\Pb}[Z]$, is the left-side $(1-\beta)$-quantile of $Z$. Formally, $\VaR_\beta^\Pb [Z]  	 = \inf \setdef{\zeta}{\Pb(Z \le \zeta) \ge 1-\beta}$. 
The \myemphc{conditional value-at-risk (CVaR)} of $Z$ at level $\beta$, denoted $\CVaR_\beta^\Pb [Z]$, is given as 
\begin{align}\label{eq:cvar-def-alt}
	\CVaR_\beta^\Pb [Z] = \inf_{t \in \real} \Bigl\{ t + \beta^{-1} \Eb^\Pb[Z - t]_+ \Bigr\},
\end{align}
where $\Eb^\Pb[\,\cdot\,]$ denotes expectation under $\Pb$. Under continuity of the cdf of $Z$ at $\VaR_\beta^\Pb[Z]$, we have 
	$\CVaR_\beta^\Pb [Z] := \Eb^\Pb[Z \ge \VaR_\beta^\Pb [Z]]$.
The parameter $\beta$ characterizes risk-averseness. When $\beta$ is close to unity, the decision-maker is risk-neutral, whereas, $\beta$ close to the origin implies high risk-averseness.

\section{Problem Statement} \label{sec:problem}

Consider the following discrete-time system:
\begin{equation}\label{sys}
    x_{t+1}  = f(x_t,u_t), 
\end{equation}
where $\map{f}{\real^{n_x} \times \real^{n_u}}{\real^{n_x}}$ defines the dynamics and $x_t\in\real^{n_x}$ and $u_t\in\real^{n_u}$ are the state and control input of the system at time $t$, respectively. The system state and control input are subject to the following deterministic constraints:
\begin{equation}
    x_t\in \XX , u_t\in \UU , \quad  \forall t\geq 0,
\end{equation}
where $\XX$ and $\UU$ are assumed to be \emph{compact} \emph{convex} sets. The aim is to solve an infinite-horizon risk-constrained optimal control problem for system~\eqref{sys} that drives the system to a target equilibrium point $x_F \in \XX$. To that end, let $\map{r}{\XX \times \UU}{\realnonnegative}$ be a continuous function that represents the \emph{stage cost}  associated to the optimal control problem. We assume that 
\begin{align}\label{eq:st-cost}
\begin{cases}
r(x_F, 0) & = 0,\\
r(x, u) & > 0, \quad \forall (x,u) \in (\XX \times \UU) \setminus \{(x_F,0)\}.
\end{cases}
\end{align}
Using this cost function, the risk-constrained infinite-horizon optimal control problem is given as
\begin{subequations}\label{eq:IHOCP}
\begin{align}
	\min\quad  &\sum_{t=0}^{\infty}r(x_{t}, u_{t})\label{eq:IHOCP-obj} 
    \\
    \text{s.t.}  \quad & x_{t+1} = f(x_t,u_t), \quad \forall t\geq 0,\label{eq:IHOCP-a}
    \\
    \quad & x_t \in \XX , u_t\in \UU , \quad  \forall t\geq 0,\label{eq:IHOCP-c}
    \\
    \quad & x_0 = x_S,\label{eq:IHOCP-b}
    \\
    \quad & \CVaR_{\beta}^\mathbb{P}\left[g(x_t,w)\right] \leq \delta, \quad \forall t \geq 0, \label{eq:IHOCP-d}
\end{align}
\end{subequations}
where $x_S \in \XX$ is the initial state and constraint~\eqref{eq:IHOCP-d} represents the risk-averseness. Here, $\CVaR$ stands for the conditional value-at-risk (see Section~\ref{subsec:cvar} for details), $w$ is a random variable with distribution $\Pb$ supported on the compact set $\WW \subset \real^{n_w}$,  $\delta > 0$ is the risk tolerance parameter, 
$\beta > 0$ is the risk-averseness coefficient, and the continuous function $\map{g}{\XX \times \WW}{\real}$  is referred to as the constraint function. The constraint~\eqref{eq:IHOCP-d} ensures that the risk associated to the state at any time, as specified using the random function $g$, is bounded by a given parameter $\delta$. More generally, the constraint can be perceived as a safety specification for system~\eqref{sys} under uncertain environments. 

The infinite-horizon problem~\eqref{eq:IHOCP} is difficult to solve in general due to state, input, and risk constraints. Besides, in practice, the distribution $\Pb$ is usually unknown beforehand. To tackle these challenges, we propose a data-driven iterative MPC scheme outlined in the following section.

\section{Distributionally Robust Risk-constrained Iterative MPC}
In this section, we provide an iterative strategy for solving the infinite-horizon optimal control problem~\eqref{eq:IHOCP} in an approximate manner. Here, each iteration refers to an execution of the control task, that is, taking the system state from $x_S$ to $x_F$ in a safe manner. Our iterative framework is inspired by~\cite{UR-FB:17-tac} and roughly proceeds in the following manner. At the start of any iteration $j$, we have access to a finite number of samples of the uncertainty, a set of safe states, and the cost it takes to go from each of these safe states to the target. In iteration $j$, we use this prior knowledge and define an MPC scheme that constructs a safe trajectory starting at $x_S$ and ending at $x_F$. The aim of this newly generated trajectory is to possibly reduce the cost or improve the safety as compared to the previous iterations. At the end of the iteration, we update the dataset with samples gathered along the execution of the MPC scheme. Subsequently, we update the set of safe states. In the following, we make all the necessary ingredients of the iterative framework precise and later put them together in the form of  Algorithm~\ref{ag:DR_iteration}. 

\subsection{Components of the Iterative Framework}

\subsubsection{Trajectories}
Every iteration results into a trajectory. The system state and the control input at time $t$ of the $\jth$ iteration are denoted as $x_t^j$ and $u_t^j$, respectively, and the $\jth$ \emph{trajectory} is given by concatenated sets:
\begin{equation}\label{cl-traj}
	\begin{split}
		\xtraj^j& :=[x_0^j, x_1^j, \dots, x_t^j, \dots, x_{T_j}^j],\\
		\utraj^j& :=[u_0^j, u_1^j, \dots, u_t^j, \dots, u_{T_j - 1}^j].
	\end{split}
\end{equation}
We assume that all trajectories start from $x_S$, that is, $x_0^j = x_S$ for all $j \geq 1$. While our objective is to solve an infinite-horizon problem~\eqref{eq:IHOCP}, for practical considerations, we aim to find trajectories that reach the target $x_F$ in a finite number of steps. Thus, we assume that for each iteration $j$, the \emph{length} of the trajectory is finite, denoted by 
$T_j \in \naturalnumbers_{\ge1}$. Throughout the paper, whenever we mention trajectory of states, we implicitly mean that there exists a feasible control sequence that makes this trajectory of states possible. 

\subsubsection{Data and Ambiguity Sets}
At the start of iteration $j$,  a \emph{dataset} $\Data^{j-1} := \setr{\what_1, \dots, \what_{N_{j-1}}} \subset \WW$ of $N_{j-1}$ i.i.d. samples of the uncertainty $w$ drawn from $\Pb$ is available. Here, the index $j-1$ indicates the samples collected till iteration $j-1$. We assume that we collect one sample per time-step of each iteration and so the number of samples available for iteration $j+1$ are $N_j = N_{j-1} + T_j$. Our aim is to use the dataset $\Data^{j-1}$ to enforce the risk constraint~\eqref{eq:IHOCP-d} in an appropriate sense for the trajectory generated in the $\jth$ iteration. To this end, we adopt a distributionally robust approach. That is, we generate a set of distributions, termed \emph{ambiguity set}, that contains the underlying distribution $\Pb$ with high probability. We then enforce the risk constraint~\eqref{eq:IHOCP-d} for all distributions in the ambiguity set. To put the notation in place, assume that given a \emph{confidence parameter} $\zeta \in (0,1)$, we have access to a map   $\setmap{\Db}{\WW_\infty}{\PP(\WW)}$ such that given any set of $N$ i.i.d samples $\Data_N = \setr{\what_1, \dots ,\what_N}$ the set of distributions $\Db(\Data_N)$ contains $\Pb$ with confidence $\zeta$. In the definition of the map, the domain is $\WW_\infty = \cup_{i=1}^\infty \WW^i$ and $\PP(\WW)$ denotes the set of all distributions supported on $\WW$. We assume that $\Db$ always leads to a closed and nonempty ambiguity set.  We term $\Db$ as the \emph{ambiguity set generating map}. Given $\Db$, our strategy is to set the ambiguity set used for iteration $j$ as $\DD^{j-1}:=\Db(\Data^{j-1})$. The assumption on $\Db$ imply that $\DD^{j-1}$ is \emph{$(\zeta,\Pb^{\abs{\Data^{j-1}}})$-reliable}, that is, 
\begin{align}\label{eq:amb_def-n}
	\Pb^{\abs{\Data^{j-1}}} \left( \Pb\in\DD^{j-1} \right)\geq \zeta.
\end{align}
The above property implies that for the MPC scheme related to the $\jth$ iteration, if we impose the risk constraint~\eqref{eq:IHOCP-d} for all distributions in $\DD^{j-1}$, then the generated trajectory will satisfy the risk constraint with at least probability $\zeta$. Ideally, we must aim to find trajectories that satisfy~\eqref{eq:IHOCP-d}. However, when only limited data regarding the uncertainty is known, one can only enforce such a constraint in a probabilistic manner and the above definition aims to capture this feature. 

\subsubsection{Cost-to-go}
The cost-to-go from time $t$ for the trajectory $(\xtraj^j,\utraj^j)$ generated in iteration $j$, is denoted as:
\begin{align}\label{eq:to-go}
	\costgo{t}{\infty}^j &:= \textstyle \sum_{k=t}^\infty r(x_k^j, u_k^j).
\end{align}
Setting $t=0$ in~\eqref{eq:to-go} gives us the cost of the $\jth$ iteration as $\costgo{0}{\infty}^j$, that measures the performance of the controller in that iteration.
For every time-step $t \ge T_j$, we assume that the system remains at $x_F$ and the control input is zero. Thus, the infinite sum in~\eqref{eq:to-go} is well-defined due to~\eqref{eq:st-cost}.

\subsubsection{Sampled safe set}
The main advantage of the iterative scheme is that it allows data to be gathered and state-space to be explored in an incremental manner. That is, we keep track of all samples from past iterations (discussed above) and we also maintain a set of safe states (along with the minimum cost that it takes to go to the target from them) that were visited in the previous iterations. These safe states are used in an iteration as terminal constraints in the MPC scheme (as proposed in~\cite{UR-FB:17-tac}).

In iteration $j$, the risk constraint~\eqref{eq:IHOCP-d} is imposed for all distributions in $\DD^{j-1}$ in the finite-horizon optimal control problem solved in the MPC scheme (see Section~\ref{sec:dr-finite-horizon}). Thus, due to~\eqref{eq:amb_def-n}, the trajectory $(\xtraj^j,\utraj^j)$ is \emph{$(\zeta,\Pb^{\abs{\Data^{j-1}}})$-safe}, that is 
\begin{align}\label{eq:traj-safety}
	\Pb^{\abs{\Data^{j-1}}} \left( \CVaR_{\beta}^\Pb \left[g(x^j_t,w)\right] \le \delta\right) \ge \zeta
\end{align}
for all $t \in [T_j]$. Note that $\xtraj^j$ is safe with respect to the dataset $\Data^{j-1}$. However, since the next iteration $j+1$ is built considering safety with respect to the dataset $\Data^j$, all previously generated trajectories need to be $(\zeta,\Pb^{\abs{\Data^{j}}})$-safe to be considered as the set of safe states in iteration $j+1$. In particular, the \emph{sampled safe set} obtained at the end of iteration $j$ and to be used in iteration $j+1$, denoted $\mathcal{S}^{j} \subseteq \until{j} \times \XX \times \realnonnegative$, is defined recursively as
\begin{align}\label{eq:SS-update-gen}
	\mathcal{S}^{j} = \mathbb{S}^j \Bigl(\mathcal{S}^{j-1} \cup \setr{(j,x_t^{j},\costgo{t}{\infty}^{j})}_{t=1}^{T_j} \Bigr).  
\end{align}
In the above expression, the set $\setr{(j,x_t^{j},\costgo{t}{\infty}^{j})}_{t=1}^{T_j}$ collects the set of states visited in iteration $j$, along with the associated cost-to-to. The counter $j$ is maintained in this set to identify the iteration to which a state with a particular cost-to-go is associated with. The set $\mathcal{S}^{j-1}$ is the sampled safe set used in iteration $j$. The map $\mathbb{S}^j$ only keeps the states that are safe with respect to the new data set $\Data^j$. This aspect of our method is different from~\cite{UR-FB:17-tac} where explored sates are safe for all future iterations. The exact procedure that defines $\mathbb{S}^j$ is given in our algorithm.

For ease of exposition, we define maps $\Pit{\cdot}$, $\Pis{\cdot}$, and $\Pic{\cdot}$, such that, given a safe set $\mathcal{S}$,  $\Pit{\mathcal{S}}$, $\Pis{\mathcal{S}}$, and $\Pic{\mathcal{S}}$ return the set of all trajectory indices, states, and cost-to-go values that appear in $\mathcal{S}$, respectively.  The following assumption is required to initialize our iterative procedure with a nonempty sampled safe set. 

\begin{assumption}\longthmtitle{Initialization with robust trajectory}\label{assump:safeset}
	Before starting the first iteration, the sampled safe set $\mathcal{S}^0$ contains the states of a finite-length robustly safe trajectory $\xtraj^0$ that starts from $x_S$ and reaches $x_F$. This means that the trajectory $\xtraj^0$ in $\mathcal{S}^0$ robustly satisfies all constraints of problem~\eqref{eq:IHOCP}, that is, 
		$x\in\XX,\,g(x, w) \leq \delta$ for all $w \in \WW$,
	and all $x \in \Pis{\mathcal{S}^0}$.\oprocend
\end{assumption}

\subsubsection{Minimum Cost-to-go}

The sampled safe set $\mathcal{S}^j$ keeps track of the cost-to-go associated with each state in the set. However, a state can appear in more than one trajectory.  For such cases, we need to maintain the minimum cost-to-go associated with a state. To this end, given the safe set $\mathcal{S}^j$ obtained at the end of iteration $j$, we define the associated minimum cost-to-go map as 

\begin{align}\label{eq:Qj}
	Q^j(x):=\begin{cases}
		\min\limits_{J\in F^j(x)}J, &\quad x\in\Pis{\mathcal{S}^j},
		\\
		+\infty, &\quad x\notin\Pis{\mathcal{S}^j},
	\end{cases}
\end{align}
where
\begin{align}\label{def:Fj}
	F^j(x) = \setdef{\costgo{t}{\infty}^i }{ \Pis{\left\{(i, x^i_t, \costgo{t}{\infty}^i)\right\}} = \{x\},&\nonumber
		\\ 
		(i, x^i_t, \costgo{t}{\infty}^i) \in \mathcal{S}^j}&.
\end{align}
In the above expression the set $F^j(x)$ collects all cost-to-go values associated to the state $x \in \Pis{\mathcal{S}^j}$.  The function $Q^j$ then finds the minimum among these.

\subsubsection{DR Risk-constrained Finite-Horizon Problem}\label{sec:dr-finite-horizon}
Given the above described elements, we now present the finite-horizon optimal control problem solved at each time-step of each iteration. For generality, we write the problem for generic current state $x$, sampled safe set $\overline{\mathcal{S}}$, and ambiguity set $\DDo$. Let $K \in \naturalnumbers_{\ge 1}$ be the length of the horizon and consider

\begin{equation}\label{eq:DR-RLMPC:main}
	\mathcal{J}_{(\overline{\mathcal{S}}, \DDo)} (x) := \begin{cases}
		\min & \, \,  \sum_{k=0}^{K-1}r(x_{k}, u_{k}) +\overline{Q}(x_{K}) 
		\\
		\st & \, \, x_{k+1} = f(x_{k},u_{k}),   \forall k\in[K-1],
		\\
		& \, \, x_{k} \in\XX, u_{k}\in\UU,   \forall k\in[K-1],
		\\
		& \, \, x_{0}=x,
		\\
		& \, \, x_{K}\in\Pis{\overline{\mathcal{S}}}, 
		\\
		& \, \, \sup_{\mu\in\overline{\DD}} \left[\CVaR_{\beta}^{\mu}\left[g(x_{k},w)\right]\right]\leq \delta, 
		\\
		& \qquad \qquad \qquad \qquad \forall k\in[K-1],
	\end{cases}
\end{equation}
where $\map{\overline{Q}}{\XX}{\real}$ gives the minimum cost-to-go for all states in $\overline{\mathcal{S}}$ and is calculated in a similar manner as in~\eqref{eq:Qj}. The decision variables in the above problem are $(x_0, x_1, \dots, x_K)$ and $(u_0, u_1, \dots, u_{K-1})$. The set $\overline{\mathcal{S}}$ defines the terminal constraint $x_K \in \Pis{\overline{\mathcal{S}}}$. Finally, note that the risk constraint is required to hold for all distributions in the ambiguity set $\overline{\DD}$. Thus, we refer to it as \emph{distributionally robust (DR) constraint}.  For iteration $j$ and time-step $t$, the MPC scheme solves the finite-horizon problem~\eqref{eq:DR-RLMPC:main} with $x = x_t^j$, $\overline{\mathcal{S}} = \mathcal{S}^{j-1}$, $\DDo = \DD^{j-1}$, and $\overline{Q} = Q^{j-1}$.

\subsection{The Iterative Framework}
Here, we compile the elements described in the previous section and present our iterative procedure termed \emph{distributionally robust risk-constrained iterative MPC (DR-RC-Iterative-MPC)}. The informal description is given below. 
\begin{quote}
	\emph{[Informal description of Algorithm \ref{ag:DR_iteration}]:}
	Each iteration $j \ge 1$ starts with a sampled safe set $\mathcal{S}^{j-1}$ and an ambiguity set $\DD^{j-1}$. The ambiguity set is constructed (see Line~\ref{ln:ambiguity}) using samples in dataset $\Data^{j-1}$ collected in previous iterations and the map $\Db$ that ensures~\eqref{eq:amb_def-n}. In the first step of the iteration (Line~\ref{ln:drmpc}), a trajectory $(\xtraj^j,\utraj^j)$ is generated by the $\mathtt{DR\_MPC}$ routine (described in Algorithm~\ref{ag:DR_MPC}) to which the  sampled safe set $\mathcal{S}^{j-1}$ and the ambiguity set $\DD^{j-1}$ are given as inputs. This trajectory is $(\zeta,\Pb^{\abs{\Data^{j-1}}})$-safe, that is, it satisfies~\eqref{eq:traj-safety}. The samples collected in iteration $j$ are appended to the dataset $\Data^{j-1}$ in Line~\ref{ln:drmpc} and the ambiguity set for the next iteration is constructed in Line~\ref{ln:ambiguity}. The trajectory $\xtraj^j$ along with its associated cost-to-go is appended to the sampled safe set in Line~\ref{ln:safe}. In Lines~\ref{ln:uns-traj} to~\ref{ln:ss-update}, the sampled safe set $\mathcal{S}^{j-1}$ is modified to make sure that it only contains trajectories that are $(\zeta,\Pb^{\abs{\Data^{j}}})$-safe. These steps collectively represent the map $\mathbb{S}$ defined in an abstract manner in~\eqref{eq:SS-update-gen}. 	
	 
	 The indices of trajectories present in $\mathcal{S}^{j-1}$ are maintained in the set $\II^{j-1}$. In Line~\ref{ln:uns-traj}, trajectories in $\II^{j-1} \cup \{j\}$ for which at least one state is not  $(\zeta,\Pb^{\abs{\Data^{j}}})$-safe are enumerated in the set $\Unsafe^{j}$. Accordingly, in Line~\ref{ln:M}, the set $\II^j$ is updated as trajectories in $\II^{j-1} \cup \{j\}$ that are not in $\Unsafe^j$. The states visited in these trajectories are stored in $\mathcal{S}^j$ in Line~\ref{ln:ss-update}. Finally, the minimum cost-to-go for all states in $\mathcal{S}^j$ is updated in Line~\ref{ln:Q}
\end{quote}
Note that in the above algorithm, the sampled safe set is updated in an iterative way. That is, we add the $\jth$ trajectory to $\mathcal{S}^{j-1}$ and then check safety with respect to the dataset $\Data^{j}$. In the process, we loose some trajectories in iterations $\{1,\dots, j-1\}$ that could have been $(\zeta,\Pb^{\abs{\Data^{j}}})$-safe. An alternative way would be to store all previous trajectories and check for their safety in each iteration. However, such a process would potentially slow down the algorithm.

\begin{algorithm}[htb]
	\SetAlgoLined
	\DontPrintSemicolon
	\SetKwInOut{Input}{Input}
	\SetKwInOut{Output}{Output}
	\SetKwInOut{init}{Initialize}
	\SetKwInOut{giv}{Data}
	\SetKwInOut{params}{Parameter}
	\Input{
		\xvbox{2mm}{$\mathcal{S}^0$}\quad--$\;$Initial sampled safe set \\
		\xvbox{2mm}{$\Data^0$}\quad--$\;$Initial set of samples \\
		\xvbox{2mm}{$\II^0$}\quad--$\;$Index of trajectory in $\mathcal{S}^0$
		\\
	}
	\init{
		$j \gets 1$, $\DD^0 = \Db(\Data^0)$, $\Unsafe^{0}\gets\emptyset$}
	\While{$j > 0$ }
	{
		Set $(\xtraj^{j}, \utraj^{j})\gets \mathtt{DR\_MPC}(\mathcal{S}^{j-1}, \DD^{j-1})$; $T^j\gets\mathtt{length}(\xtraj^{j})$; $\Data^{j} \gets \Data^{j-1} \cup \{\what_i\}_{i=1}^{T^j}$ \label{ln:drmpc} \;
		Set $\DD^j \gets\Db(\Data^j)$ \label{ln:ambiguity}\; 
		Set $\mathcal{S}^{j-1}\gets\mathcal{S}^{j-1} \cup \{(j, x_t^j, J^j_{t\rightarrow\infty})\}_{t=1}^{T_j}$ \label{ln:safe}\;
			Set $\Unsafe^j \gets \{ i \in (\II^{j-1}\cup \{j\} ) \, |(i, x, J) \in \mathcal{S}^{j-1},$ $\sup\limits_{\mu \in \DD^{j}} \left[\CVaR_{\beta}^{\mu} \left[g(x,w)\right]\right] > \delta \}$ \label{ln:uns-traj} 
\;
		Set $\II^{j} \gets (\II^{j-1}\cup \{j\} ) \setminus \Unsafe^{j}$ \label{ln:M} \;
		Set $\mathcal{S}^{j} \gets \setdef{(i, x, J) \in \mathcal{S}^{j-1}}{i \in \II^{j}}$ \label{ln:ss-update}\;
		Compute $Q^j(x)$ for all $x\in\Pis{\mathcal{S}^j}$ using~\eqref{eq:Qj}\label{ln:Q}\;
		Set $j\gets j+1$
	}
	
	\caption{DR-RC-Iterative-MPC} 
	\label{ag:DR_iteration} 
\end{algorithm}

Algorithm~\ref{ag:DR_iteration} calls the $\drmpc$ routine in each iteration to generate the trajectory. This procedure is given in Algorithm~\ref{ag:DR_MPC} and informally described below.
\begin{quote}
	\emph{[Informal description of Algorithm~\ref{ag:DR_MPC}]:} 
	The procedure generates a trajectory from $x_S$ to $x_F$ given a sampled safe set $\overline{\mathcal{S}}$ and an ambiguity set $\overline{\DD}$. The minimum cost-to-go function $\Qo$ is computed for $\overline{\mathcal{S}}$ using~\eqref{eq:Qj}. At time-step $t$, problem~\eqref{eq:DR-RLMPC:main} is solved with $x = x_t$. We denote the optimal solution by 
	\begin{equation}\label{eq:optSol}
		\begin{split}
			x_{\mathrm{vec},t}^{*} &= [x_{t|t}^{*}, \dots , x_{t+K|t}^{*}], 
			\\
			u_{\mathrm{vec},t}^{*} &= [u_{t|t}^{*}, \dots , u_{t+K-1|t}^{*}],
		\end{split}
	\end{equation}
	where we use the notation that $x_{t+k|t}$ represents the prediction made at time $t$ regarding the state at time $t+k$. The control action at time $t$ is set as the first element of $u_{\mathrm{vec},t}^{*}$ (Line~\ref{ln:control}) and it is appended to the trajectory $\utraj$. The state is updated and added to $\xtraj$ in Line~\ref{ln:state}. The procedure moves to the next time step with the updated state as $x_{t+1}$. 
\end{quote}

\begin{algorithm}[htb]
	\SetAlgoLined
	\DontPrintSemicolon
	\SetKwInOut{Input}{Input}
    \SetKwInOut{Output}{Output}
	\SetKwInOut{init}{Initialize}
	\SetKwInOut{giv}{Data}
	\SetKwInOut{params}{Parameter}
	\SetKwProg{Fn}{Function}{:}{}
	\SetKwFunction{FMain}{$\mathtt{DR\_MPC}$}
    \Fn{\FMain{$\overline{\mathcal{S}}, \DDo$}}{
    	\init{
    		$t\gets0$; $x_0\gets x_S$; $\xtraj \gets[x_0]$, $\utraj \gets[\,\,]$
    	}
    Set $\overline{Q}$ as minimum cost-to-go in $\overline{\mathcal{S}}$ (use~\eqref{eq:Qj})\;
    \While{$x_t\neq x_F$}{
        Solve~\eqref{eq:DR-RLMPC:main} with $x=x_t$ and obtain optimal solutions $x_{\mathrm{vec},t}^{*}$ and $u_{\mathrm{vec},t}^{*}$\;
        Set $u_t\gets u^{*}_{t|t}$; $\utraj \gets[u, u_t]$ \label{ln:control}\;
        Set $x_{t+1} \! \gets  \! f(x_t,u_t) $; $\xtraj \! \gets \! [x, x_{t+1}]$; $t \! \gets  \! t+1$ \label{ln:state}\;
        }
    \textbf{return} $(\xtraj, \utraj)$
    }
    \textbf{end}
    \caption{Distributionally robust MPC function}
    \label{ag:DR_MPC}
\end{algorithm}

The above explained MPC procedure might not terminate in finite time, thus possibly violating our assumption that all trajectories have finite length. To practically overcome this issue, we terminate the MPC scheme when the state reaches a neighborhood of the equilibrium $x_F$.

\begin{remark}\longthmtitle{Tractability}
	Note that, if $g(\cdot,w)$ is convex for every $w \in \WW$ and~\eqref{sys} is a linear system, then the risk-constraint in the infinite-horizon problem~\eqref{eq:IHOCP-d} as well as the DR risk-constraint in~\eqref{eq:DR-RLMPC:main} are convex. As a result of the latter fact, all points in the convex hull of $\Pis{\overline{\mathcal{S}}}$ satisfy the DR risk-constraint. Hence, we can replace $\Pis{\overline{\mathcal{S}}}$ with its convex hull and define the minimum cost-to-go function using Barycentric functions (see~\cite{UR-FB:17-ifac}) in the problem~\eqref{eq:DR-RLMPC:main} without affecting the safety of the resulting trajectory. By doing so, all constraints in problem~\eqref{eq:DR-RLMPC:main} are convex which eases the computational burden of solving the problem. \oprocend
\end{remark}

\begin{remark}\longthmtitle{Ambiguity sets}\label{re:ambiguity}
	The definition of the ambiguity set in our algorithm is quite general, defined using an arbitrary map $\Db$. Popular choices of data-based ambiguity sets are the ones using distance metrics such as Wasserstein, KL-divergence, $\phi$-divergence or using moment information, see~\cite{HR-SM:19-arXiv} for a survey. The reliability guarantee~\eqref{eq:amb_def-n} for a particular choice of ambiguity set is ensured by concentration of measure results. Each class of ambiguity set comes with its own pros and cons and usually one needs to seek a balance between (a) guaranteed statistical performance, (b) generality of distributions that can be handled, and (c) computational effort for handling the DR constraint in the finite-horizon problem~\eqref{eq:DR-RLMPC:main}.  We wish to explore this further in future. \oprocend 
\end{remark}

\begin{remark}\longthmtitle{Safety vs cost-performance}\label{re:tradeoff}
	The reliability parameter $\zeta$ in our framework is tunable. Meaning, if one requires high level of safety when exploring the state-space, then $\zeta$ can be selected close to unity. In that case, the ambiguity set needs to be large enough to ensure~\eqref{eq:amb_def-n} and so the DR risk-constraint turns out to be conservative. Analogously, if cost improvement is the goal, then a low value of $\zeta$ will be sufficient. Note that safety can alternatively be tuned by changing the risk-averseness parameter $\beta$. A lower value of $\beta$ would ensure more safety.  \oprocend
\end{remark}

\section{Properties of DR-RC-Iterative-MPC}

We first establish recursive feasibility of our iterative procedure given in Algorithm~\ref{ag:DR_iteration}. We also state the safety guarantee with which each trajectory is generated.

\begin{proposition}\label{prop:t-t+1}\longthmtitle{Safety and recursive feasibility of DR-RC-Iterative-MPC}
	Let Assumption~\ref{assump:safeset} hold. Then, at each iteration $j\ge1$ and time-step $t \ge 0$, the finite-horizon problem~\eqref{eq:DR-RLMPC:main} with $x = x_t^j$, $\overline{\mathcal{S}} = \mathcal{S}^{j-1}$, and $\DDo = \DD^{j-1}$ solved in the DR-RC-Iterative-MPC scheme is feasible. Further, the generated trajectory $(\xtraj^j,\utraj^j)$ is $(\zeta,\Pb^{\Data^{j-1}})$-safe.
\end{proposition}
\begin{proof}
	By Assumption~\ref{assump:safeset}, $\mathcal{S}^0$ contains a finite-length robustly safe trajectory from $x_S$ to $x_F$, denoted as, 
	\begin{align*}
		\xtraj^0 := [x_0^0, x_1^0, \dots, x_{T_0}^0], \, \text{ and } \, 
		\utraj^0 := [u_0^0, u_1^0, \dots, u_{T_0 - 1}^0],
	\end{align*} 
	where $T_0$ is the length of this trajectory. The update step for the sampled safe set (Line~\ref{ln:safe} and~\ref{ln:ss-update}) implies that 
		$\mathcal{S}^0 \subseteq \mathcal{S}^j$, for all $j\geq 1$. 
	Thus, at each iteration $j\geq1$ and $t = 0$, the first $K$ elements of $(\xtraj^0,\utraj^0)$ are valid feasible solutions to~\eqref{eq:DR-RLMPC:main}, where $x = x_S$ and $(\overline{\DD},\overline{\mathcal{S}}) = (\DD^{j-1},\mathcal{S}^{j-1})$.
	Our next step is to show that, for each iteration, feasibility at time $t$ implies feasibility at time $t+1$. The proof then follows by induction.
	
	Assume that the optimization problem~\eqref{eq:DR-RLMPC:main} is feasible at iteration $j$ and time $t$ for $x = x_t^j$. Denote the optimizer as
	\begin{equation}\label{eq:optSol-j}
			\begin{split}
				x_{\mathrm{vec},t}^{*,j} &= [x_{t|t}^{*,j}, \dots , x_{t+K|t}^{*,j}],
				\\
				u_{\mathrm{vec},t}^{*,j} &= [u_{t|t}^{*,j}, \dots , u_{t+K-1|t}^{*,j}].
			\end{split}
	\end{equation}
	By applying the first element of $u_{\mathrm{vec},t}^{*,j}$ to system~\eqref{sys}, the new state is determined as $x_{t+1}^j = f(x_{t}^j, u_{t|t}^{*,j})$. Moreover, we have $x_{t+1}^j = x_{t+1|t}^{*,j}$. 
	Due to constraints in~\eqref{eq:DR-RLMPC:main}, we have $x_{t+K|t}^{*,j} \in \mathcal{S}^{j-1}$. Recall that due to Line~\ref{ln:ss-update} of the algorithm, $\mathcal{S}^{j-1}$ contains trajectories that are $(\zeta,\Pb^{\Data^{j}})$-safe. Thus, there exists a trajectory starting at $x_{t+K|t}^{*,j}$ given as $\tilde{\xtraj} =  [x_{t+K|t}^{*,j}, \tilde{x}_{K+1}, \dots, \tilde{x}_{K+T}]$, $\tilde{\utraj} = [\tilde{u}_{K}, \tilde{u}_{K+1}, \dots, \tilde{u}_{K+T-1}]$, 
	such that $\tilde{x}_{K+T} = x_F$ and all points in $\tilde{\xtraj}$ are in $\Pis{\mathcal{S}^{j-1}}$. Using the above trajectory, the following one is a feasible solution to~\eqref{eq:DR-RLMPC:main} for time-step $t+1$, that is, when $x=x_{t+1}^j$: 
	\begin{align*}
		& [x_{t+1|t}^{*,j}, x_{t+2|t}^{*,j}, \dots, x_{t+K|t}^{*,j}, \tilde{x}_{K+1}], 
		\\
		& [u_{t+1|t}^{*,j}, u_{t+2|t}^{*,j}, \dots, u_{t+K-1|t}^{*,j}, \tilde{u}_{K}].
	\end{align*}
	This completes the proof of the first part. The safety of the trajectory $(\xtraj^j,\utraj^j)$ follows from the constraint in~\eqref{eq:DR-RLMPC:main} and the reliability assumption on the ambiguity set~\eqref{eq:amb_def-n}.
\end{proof}

Next, we show that each trajectory generated by Algorithm~\ref{ag:DR_MPC} converges asymptotically to $x_F$.
\begin{proposition}\label{pr:attr_xf}\longthmtitle{Convergence of $\drmpc$}
	Let Assumption~\ref{assump:safeset} hold. Then, for each iteration $j \ge 1$ of the DR-RC-Iterative-MPC procedure, the trajectory $(\xtraj^j,\utraj^j)$ generated by $\drmpc$ satisfies $x^j_t \to x_F$ as $t \to \infty$. 
\end{proposition}
\begin{proof}
	To show the result we will make use of Proposition~\ref{prop:xbar_conv}. Note that the trajectory belongs to the compact set $\XX$ and is generated by the discrete-time system $x^j_{t+1} = f(x_t^{j}, u^{*,j}_{\mathrm{vec},t})$, where $u^{*,j}_{\mathrm{vec},t}$ is the first component of the control input sequence obtained by solving the finite-horizon optimal control problem~\eqref{eq:DR-RLMPC:main} with $x = x_t^j$ (see~\eqref{eq:optSol-j} for more details). Note that the control input $u^{*,j}_{\mathrm{vec},t}$ is a function of the state $x_t^j$ only and so, the dynamics can be implicitly written in the form~\eqref{eq:gen-sys}. Next, consider the function $\map{\JJ_{(\mathcal{S}^{j-1},\DD^{j-1})}}{\XX}{\realnonnegative}$ obtained by replacing $(\overline{\mathcal{S}},\DDo)$ with $(\mathcal{S}^{j-1},\DD^{j-1})$ in the definition of $\JJ_{(\overline{\mathcal{S}},\DDo)}$ given in~\eqref{eq:DR-RLMPC:main}. For brevity, we use the shorthand $\JJ^{j-1}$. Our aim is to show that $\JJ^{j-1}$ acts as the Lyapunov candidate $V$. First note that $\JJ^{j-1}(x_F) = 0$ as one can apply zero input at the equilibrium and stay there, thus accumulating no cost. For any $x \not = x_F$, we have $\JJ^{j-1}(x) \ge \min_{u \in \UU} r(x,u) > 0$. Thus, in our case $x_F$ acts as the point $x^*$ in Proposition~\ref{prop:xbar_conv}. To conclude the proof, we show that $\JJ^{j-1}$ satisfies~\eqref{eq:V-lyap}. Pick any $x \in \XX$ and let $(\xtraj^*,\utraj^*) \in \XX^{K+1} \times \UU^K$ be the optimal trajectory obtained by solving the finite-horizon problem~\eqref{eq:DR-RLMPC:main} with the constraint that the first state in the trajectory is $x^*_1 = x$. We have
	\begin{align*}
		\JJ^{j-1}(x_1^*) = r(x^*_1,u^*_1) + \sum_{k=2}^K r(x^*_k,u^*_k) + Q^{j-1}(x^*_{K+1}).
	\end{align*}
    By definition of $Q^{j-1}$, there exists a trajectory $j^* \in \until{j-1}$ and a time $t^*$ such that $x^*_{K+1} = x^{j^*}_{t^*}$ and 
    	$Q^{j-1}(x^*_{K+1}) = \sum_{t=t^*}^\infty r(x^{j^*}_t,u^{j^*}_t)$.
	Substituting this expression in the above relation for $\JJ^{j-1}$, gives 
	\begin{align}
		\JJ^{j-1}(x_1^*) & = r(x^*_1,u^*_1) + \sum_{k=2}^K r(x^*_k,u^*_k) + \sum_{t=t^*}^\infty r(x^{j^*}_t,u^{j^*}_t) \notag
		\\
		& = r(x^*_1,u^*_1) + \textstyle \sum_{k=2}^K r(x^*_k,u^*_k) + r(x^{j^*}_{t^*},u^{j^*}_{t^*}) \notag
		\\
		& \qquad \qquad + \textstyle \sum_{t=t^*+1}^\infty r(x^{j^*}_t,u^{j^*}_t) \label{eq:J-ineq-2}
	\end{align}
	Note that $x_{t^*+1}^{j^*}$ belongs to the set $\mathcal{S}^{j-1}$ and so 
	\begin{align}\label{eq:Q-bound}
		Q^{j-1}(x_{t^*+1}^{j^*}) \le \textstyle \sum_{t=t^*+1}^\infty r(x^{j^*}_t,u^{j^*}_t),
	\end{align}
	and the finite-horizon trajectory $\xtraj^+ = [x^*_{2}, \dots, x^*_{K+1}, x^{j^*}_{t^*+1}]$, $\utraj^+ = [u^*_{2}, \dots, u^{j^*}_{t^*}]$ is a feasible solution for~\eqref{eq:DR-RLMPC:main} with ${x = x_2^*}$. 
	Using optimality, we have
	\begin{align*}
		& \JJ^{j-1}(x^*_2) \le \sum_{k=2}^K r(x^*_k,u^*_k) + r(x^*_{K+1},u^{j^*}_{t^*}) + Q^{j-1}(x^{j^*}_{t^* +1})
		\\
		& \quad \le \sum_{k=2}^K r(x^*_k,u^*_k) + r(x^{j^*}_{t^*},u^{j^*}_{t^*}) + \sum_{t=t^*+1}^\infty r(x^{j^*}_t,u^{j^*}_t),
	\end{align*}
	where we used~\eqref{eq:Q-bound} in the above inequality. 
	Using the above condition in~\eqref{eq:J-ineq-2} yields 
		$\JJ^{j-1}(x_1^*) \ge \JJ^{j-1}(x^*_2) + r(x^*_1,u^*_1)$.
	Since $r$ is a continuous function satisfying~\eqref{eq:st-cost}, this establishes the inequality~\eqref{eq:V-lyap} and so completes the proof. 
\end{proof}

Previous results were aimed at the guarantees that each iteration of our Algorithm enjoys. Next, we examine the performance of our algorithm across iterations.
	
	\begin{proposition}\longthmtitle{Guarantee across iterations for DR-RC-Iterative-MPC}\label{pr:across-iter}
		For the DR-RC-Iterative-MPC procedure, if $\DD^{j} \subset \DD^{j-1}$ for some iteration $j \ge 1$, then we have 
		\begin{align}\label{eq:SS-over-iter}
			\mathcal{S}^j = \mathcal{S}^{j-1} \cup \setr{(j,x_t^{j},\costgo{t}{\infty}^{j})}_{t=1}^{T_j}.
		\end{align}
	As a consequence, $\mathcal{S}^{j-1}\subseteq\mathcal{S}^j$. In addition, if the function $\JJ_{(\mathcal{S}^{j-1},\DD^{j-1})}$ is continuous at $x_F$, then   $\costgo{0}{\infty}^j \le \costgo{0}{\infty}^{j-1}$. 
	\end{proposition}

We omit the proof for space reasons. The assumption that $\DD^j \subset \DD^{j-1}$ is a difficult one to impose, in general. In future, we would like to explore scenarios where this can be ensured at least with high probability if not almost surely.
\section{Simulation}\label{sec:sims}

We demonstrate the performance of Algorithm~\ref{ag:DR_iteration} via a motion planning task for a mobile robot where the environment includes a randomly moving obstacle. As iterations progress, more data is collected and the safe set is gradually expanded. As a result, the cost-performance of the system improves while ensuring the required safety.
\subsubsection{Setup}
Consider the following model for a circular mobile robot navigating in a 2D environment:
\begin{align*}
    x_{t+1} & = \begin{bmatrix}
            1 & 0 & 1 & 0 \\
            0 & 1 & 0 & 1 \\
            0 & 0 & 1 & 0 \\
            0 & 0 & 0 & 1 
    \end{bmatrix}x_t + \begin{bmatrix}
        0 & 0 \\
        0 & 0 \\
        1 & 0 \\
        0 & 1 \\
    \end{bmatrix}u_t.
\end{align*}
Here, the state $x = [z,y,v_z,v_y]^\top$ contains the position $(z,y)$ of the center of mass of the robot and its velocity in $z$ and $y$ directions. The input $u = [a_z,a_y]^\top$ consists of the acceleration in $z$ and $y$ directions. The objective of this problem is to steer the agent from the initial point $x_S = [0, 0, 0, 0]^\top$ to the target point $x_F = [5, 3, 0, 0]^\top$ while constraining the risk of colliding with a square obstacle of length $\ell_\OO=0.4$ that moves randomly around the point $[2, 2]^\top$. Specifically, the position of the obstacle in each time-step is given by $o_t = \begin{bmatrix}2 & 2\end{bmatrix}^\top + \begin{bmatrix}\frac{1}{\sqrt{2}} & -\frac{1}{\sqrt{2}}\end{bmatrix}^\top w_t$,
where $w_t\in\real$ is the uncertainty. We assume that $w_t$ is defined by the $\text{Beta-binomial}(15, 10, 15)$ distribution supported on the set of fifteen points  
	$\setdef{-0.5 + i/14}{i \in \until{14}}$.
Since a small number of samples are usually available in practice, we start with only $N_0=5$ samples. We assume that in each time-step of each iteration, the obstacle's position is observable, which forms the dataset of samples.
Given position $o_t$, the region of the environment occupied by the obstacle is given as
\begin{align*}
	\OO_t = \setdef{(z,y) \in \real^2}{o_t-\tfrac{\ell_\OO}{2}  \ones_2 \le [z \, \, y]^\top  \le o_t + \tfrac{\ell_\OO}{2}  \ones_2},
\end{align*}
where $\ones_2 = [1, 1]^\top$. The stage cost is quadratic, given as ${r(x_t, u_t)=(x_F-x_t)^{\top}Q(x_F-x_t) + u_t^{\top}Ru_t}$, where ${Q=\text{diag}(1, 1, 0.01, 0.01)}$ and ${R=\text{diag}(0.01, 0.01)}$. Note that $r$ satisfies the condition~\eqref{eq:st-cost}. 
The safe set $\mathcal{S}^0$ and its respective terminal cost $Q^0$, that are required for initializing the first iteration of Algorithm~\ref{ag:DR_iteration}, are generated using an open-loop controller that drives the agent to the target while being far away from the obstacle, see Figure~\ref{fig:trajs}. We execute the algorithm for $20$ iterations.
 The prediction horizon is $K=5$. We use $\beta = 0.05$ as the risk-averseness coefficient and $\delta = 0.02$ as the right-hand side of the risk constraint.

We consider ambiguity sets defined using the total variation distance. For discrete distributions $P, Q \in \Delta_{\abs{\WW}}$ 
supported on a finite set $\WW$, the total variation distance between them is defined as $\delta(P, Q) = \frac{1}{2}\norm{P-Q}_1$. Given $N$ i.i.d samples $\setr{\what_1, \dots, \what_N}$ of the uncertainty, the empirical distribution is given by the vector $\widehat{\mathbb{P}}^{N} := (p_i^{N})_{i=1}^{\abs{\WW}}$, where $p_i^N = (\text{frequency of }w_i \in \WW \text{ in the dataset})/N$. Using this definition, we consider the ambiguity sets of the form 
	\begin{align}\label{eq:sim-ambiguity}
		\DD = \setdef{\mu \in \Delta_{\abs{\WW}}}{\delta(\mu, \widehat{\mathbb{P}}^{N})\le \theta},
	\end{align}
where $\theta \ge 0$ is the radius. Given a dataset and a confidence level $\zeta \in (0,1)$, one can tune the radius to obtain a reliability bound as~\eqref{eq:amb_def-n}. See~\cite{MS-PS-PP:2019} for further details. We run our experiments for different values of $\theta$ and check the safety and performance of our method.

According to the objective of avoiding collision in  this experiment, the constraint function $g$ is given as the distance between the agent and the safe region $\YY_t$ determined by excluding the instantaneous position of the obstacle from the environment. More precisely, $g(x, w) = \min_{a\in\YY}\norm{Cx-a}$, where $\YY := \real^2\textbackslash\OO$, the set $\OO$ is determined by the uncertainty $w$, and $C$ is chosen such that $Cx=\left[z, y\right]^\top$. Taking benefit of the square shape of the obstacle, we use the simplified representation of function $g$ provided in~\cite{AH-IY:20-ICRA} which is $g(x, w) = \min_{j\in[3]}\left\{\frac{\left[d_j+h_j^\top(Cx - w)\right]_+}{\norm{h_j}}\right\}$, where $h_j$ and $d_j$ represent the outward normal and the position of one of the constraints defining the obstacle set (see \cite[Lemma 1]{AH-IY:20-ICRA} for details).
Note that $g$, and so the distributionally robust constraint~\eqref{eq:DR-RLMPC:main} are non-convex, and non-trivial to handle. We use the following reformulation of the distributionally robust constraint for the ambiguity set~\eqref{eq:sim-ambiguity} to write the optimization~\eqref{eq:DR-RLMPC:main} in a finite-dimensional form:
\begin{align*}
	&\sup_{\mu\in \DD} \left[\CVaR_{\beta}^{\mu}\left[g(x,w)\right]\right] = 
	\\
	&\begin{cases} \nonumber\inf \quad 2\lambda \theta + \eta + \nu + \textstyle \sum_{\ell = 1}^L (\gamma_{1_\ell}-\gamma_{2_\ell})p_\ell^{N}
	\\
	 \text{s.t.}  \quad \beta( 
	\gamma_{1_\ell}-\gamma_{2_\ell}+\nu) \geq [g(x,w_\ell)-\eta]_+, \quad\forall\ell\in[L],\\
	  \quad \quad
	\gamma_{1_\ell}+\gamma_{2_\ell} \leq \lambda, \quad\forall\ell\in[L],\\
	 \quad \quad
	\eta, \nu\in\real, \quad \lambda, \gamma_{1_\ell}, \gamma_{2_\ell} \in \realnonnegative ,\quad\forall\ell\in[L],
	\end{cases}
\end{align*}
where $L=\abs{\WW}$. The above reformulation is different from the one given in~\cite{AH-IY:20-ICRA} as there ambiguity sets are defined using the Wasserstein metric. We have omitted the proof of the above relation due to space reasons. The optimization problem~\eqref{eq:DR-RLMPC:main}, considering the above-mentioned reformulation is implemented in GEKKO~\cite{LB-DH-RM-JH:2018} using APOPT solver.
\subsubsection{Results}
The resulted trajectories for different ambiguity set sizes are presented in Figure~\ref{fig:trajs}.
        \begin{figure*}
        \centering
        \begin{subfigure}[b]{0.32\linewidth}
            \centering
            \includegraphics[width=\linewidth]{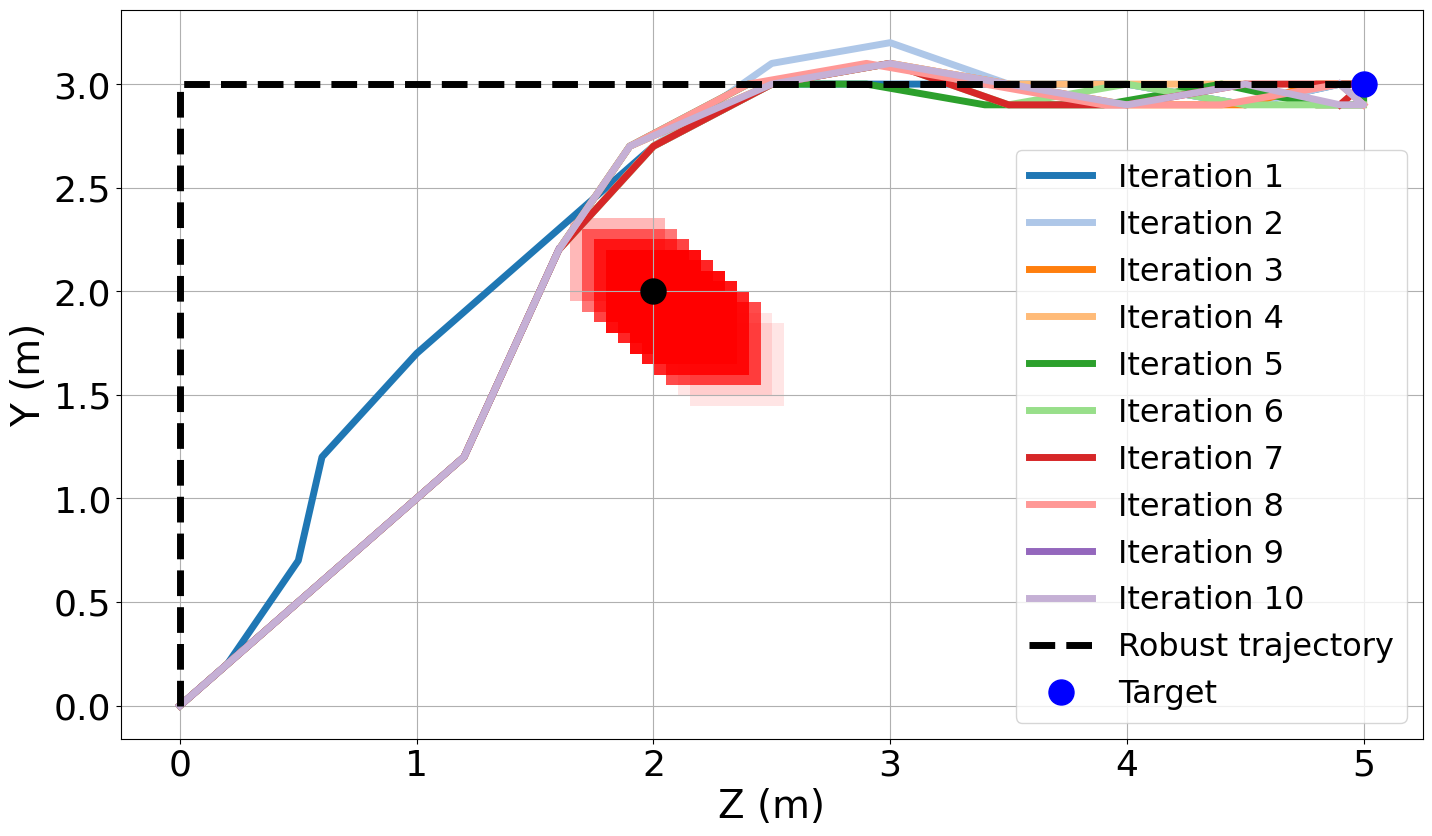}
            \caption[Network2]
            {{\small $\theta$ = $5\times10^{-6}$}}    
            \label{fig:trajs_a}
        \end{subfigure}
        \qquad
        \begin{subfigure}[b]{0.32\linewidth}  
            \centering 
            \includegraphics[width=\linewidth]{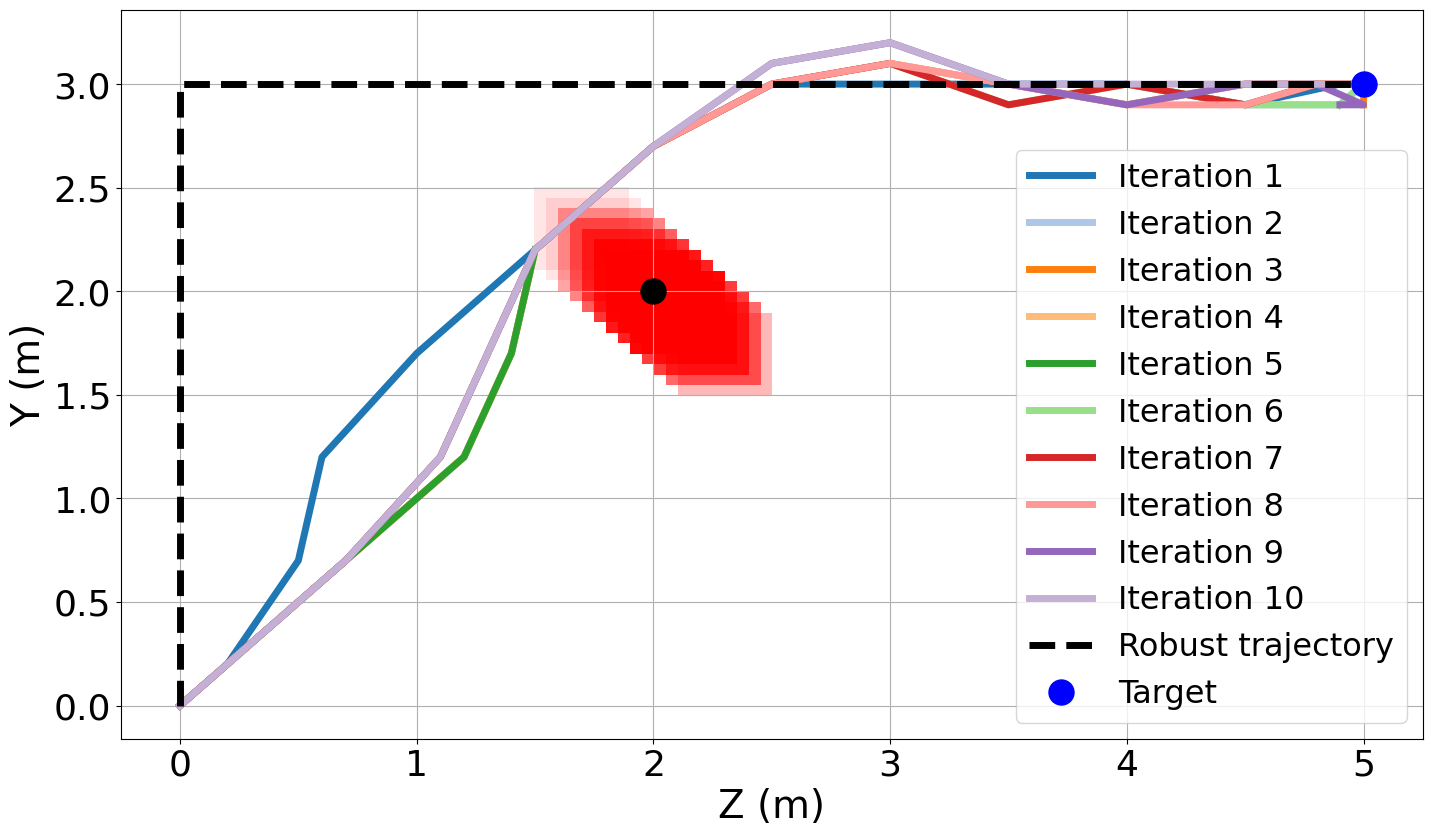}
            \caption[]
            {{\small $\theta$ = $5\times 10^{-4}$}}
            \label{fig:trajs_b}
        \end{subfigure}
        \vskip\baselineskip
        \begin{subfigure}[b]{0.32\linewidth}   
            \centering 
            \includegraphics[width=\linewidth]{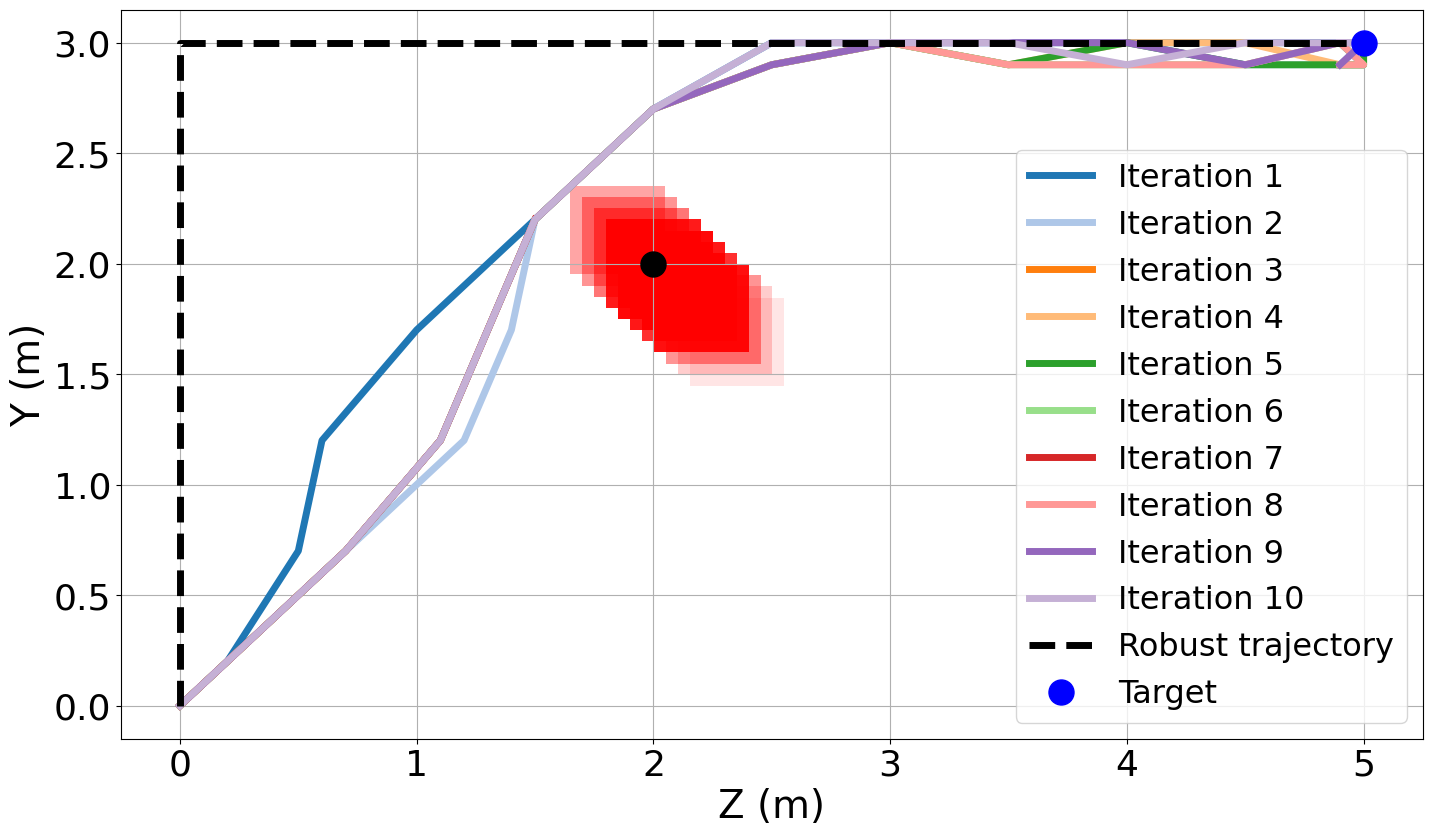}
            \caption[]
            {{\small $\theta$ = $5\times 10^{-2}$}}    
            \label{fig:trajs_c}
        \end{subfigure}
        \qquad
        \begin{subfigure}[b]{0.32\linewidth}   
            \centering 
            \includegraphics[width=\linewidth]{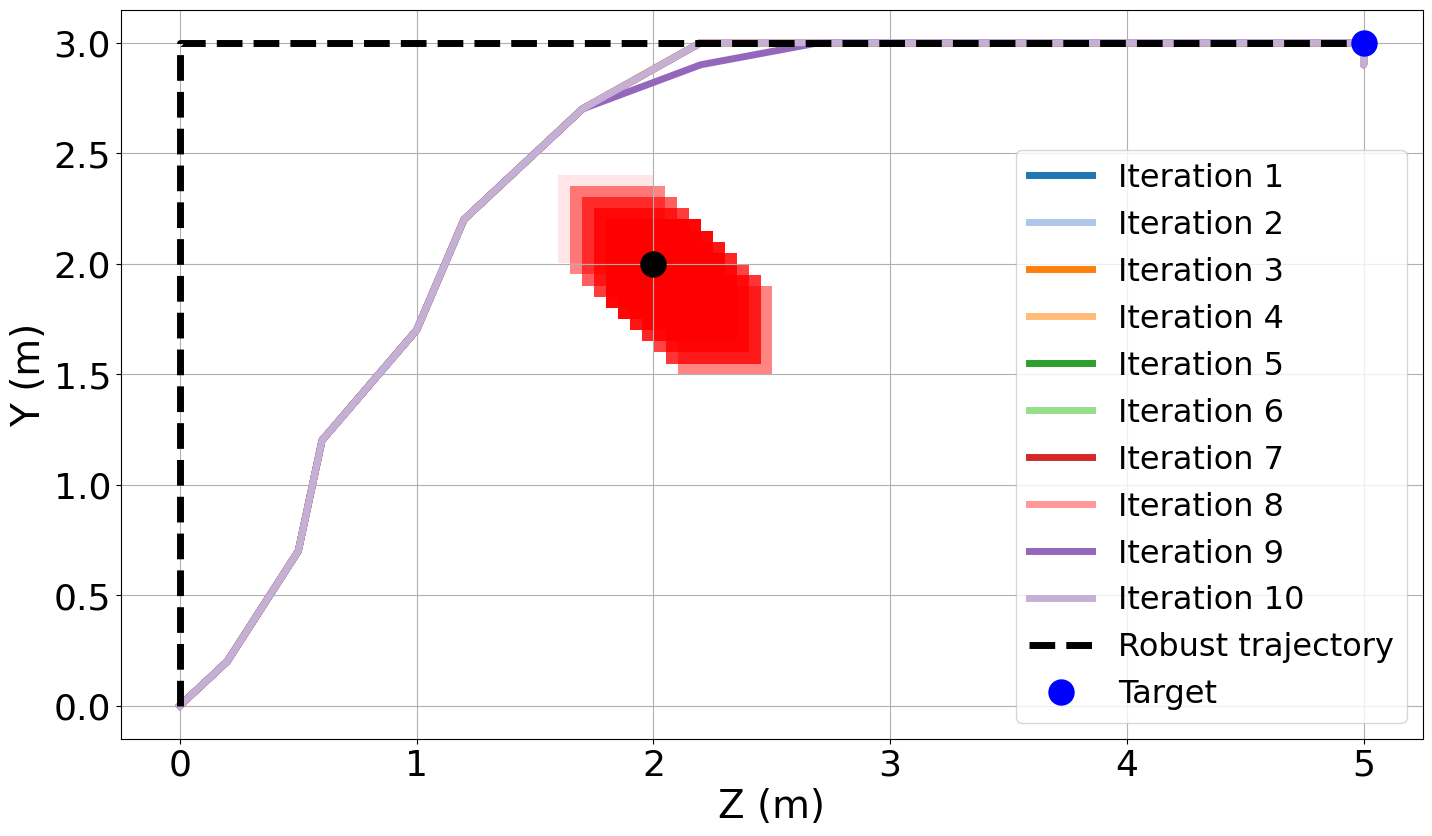}
            \caption[]
            {{\small $\theta$ = 0.5}}    
            \label{fig:trajs_d}
        \end{subfigure}
        \caption{\footnotesize Plots illustrating the application of the DR-RC-Iterative-MPC procedure for the task of navigating the mobile robot in an environment with an uncertain obstacle (see Section~\ref{sec:sims} for details). We consider four different radii for the ambiguity set and for each case, the radius does not change over the iterations. The robust trajectory (dashed black line) is the same for all cases. Each realization of the obstacle is plotted with a shaded red square. As observed, the trajectories become more conservative as the radius of the ambiguity set increases.} 
        \vspace*{-3ex}
        \label{fig:trajs}
    \end{figure*}
    In the first iteration, due to the constraint $x_{K}\in\Pis{\mathcal{S}^0}$, for different values of $\theta$, the system follows a similar trajectory. However, differences become more apparent as time-step progresses. For small ambiguity sets, the trajectories are closer to the obstacle. For larger ones, the algorithm becomes more conservative to the extent that for $\theta = 0.5$ the agent stops exploring and is only concerned about safety.  There is a noteworthy observation in Figure~\ref{fig:trajs_b}, that is, trajectories get closer to the obstacle in the first few iterations but as more data is collected, the safe set gets refined in 
    later iterations and the robot deviates from the obstacle more strongly. Finally in Figure~\ref{fig:collision} we underline the impact of the size of the ambiguity set cost-performance and safety. As shown in the figure, smaller ambiguity sets provide cost-efficient trajectories while they also increase the probability of colliding with the obstacle. 
    As a result, if an appropriate value is chosen for the radius of the ambiguity set, the algorithm is able to provide an acceptable level of safety even using small number of samples. 
\begin{figure}
    \centering
    \includegraphics[width=0.75\linewidth]{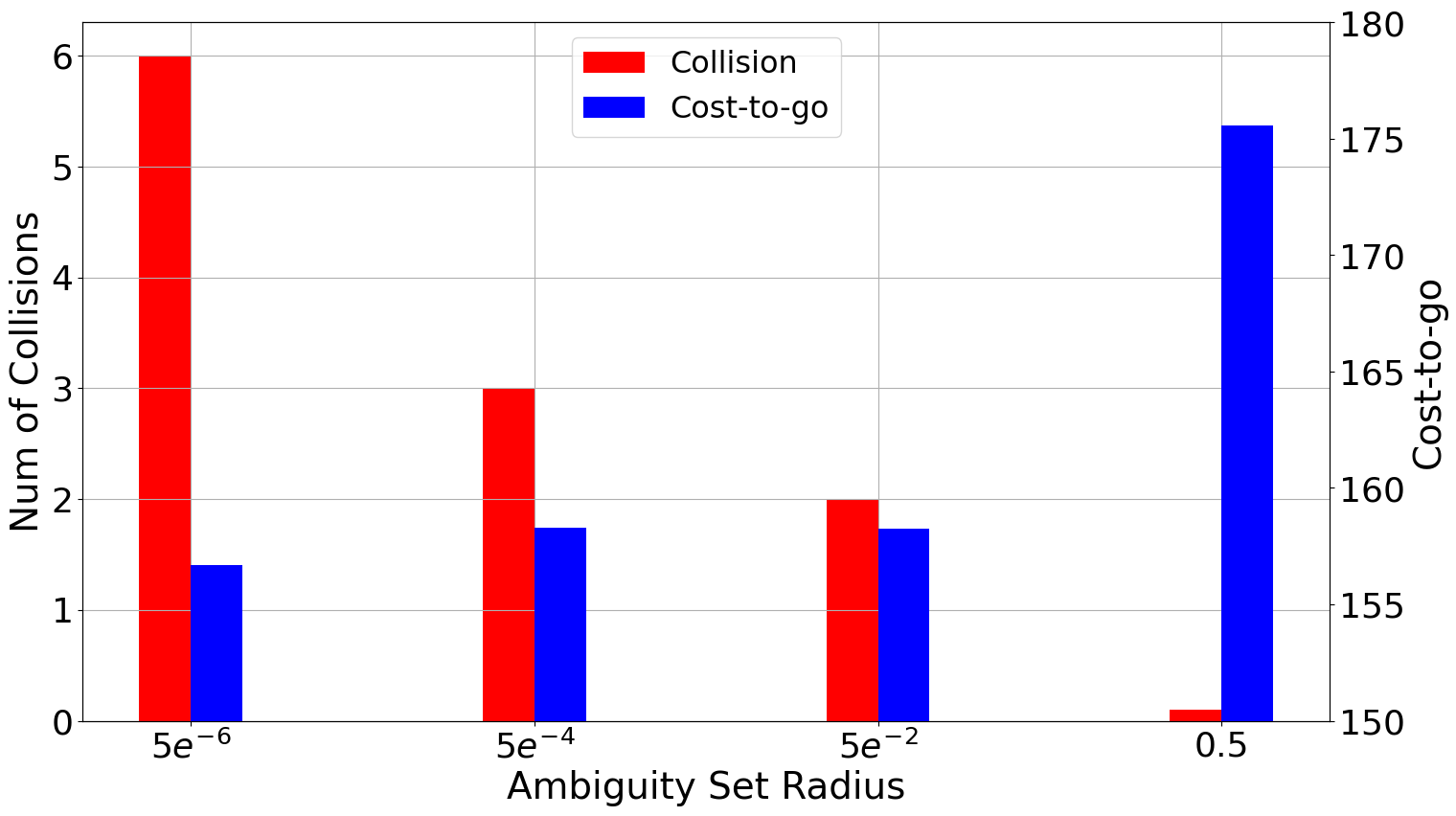}
    \caption{\footnotesize The effect of the size of the ambiguity set on safety and performance. The red block represents the number of iterations (out of $20$) in which the trajectory collides with the obstacle at least once. The blue block depicts the iteration cost of the collision-free iteration that has the highest index.}
    \vspace*{-3ex}
    \label{fig:collision}
\end{figure}

\section{Conclusions}

We considered a risk-constrained infinite-horizon optimal control problem and designed an iterative MPC-based scheme to solve it. Our procedure approximated the risk constraints using their data-driven distributionally robust counterparts. Each iteration in our method generated a trajectory that is provably safe and that converges to the equilibrium asymptotically. Lastly, we implemented our algorithm to find a risk-averse path for a mobile robot that is in an environment with uncertain obstacle. Several ideas need to be explored in the future. First, we wish to determine conditions under which the MPC scheme in each iteration converges to the equilibrium in finite time. Second, we aim to analyze the convergence of the iterative procedure. Specifically, to answer how one should adapt the ambiguity sets such that the iterations converge to the optimal solution of the infinite-horizon problem. Third, with online implementation as goal, we plan to explore the computational and statistical guarantees that various ambiguity sets have.

\bibliographystyle{ieeetr}

\appendix
\renewcommand{\theequation}{A.\arabic{equation}}
\renewcommand{\thetheorem}{A.\arabic{theorem}}
\renewcommand{\theproposition}{A.\arabic{proposition}}

The following result on attractivity of the equilibrium point of a discrete-time system aids us in showing convergence of our iterative $\drmpc$ scheme. The proof follows standard Lyapunov arguments but the exact result is not available in the literature. We provide the proof here for completeness.

\begin{proposition}\label{prop:xbar_conv}\longthmtitle{Attractivity of discrete-time system}
	Consider the system
	\begin{align}\label{eq:gen-sys}
		x_{t+1} = f(x_t), \quad x_0 \in \XX \subset \real^n,
	\end{align}
	where $\map{f}{\XX}{\XX}$ and $\XX$ is a compact set. Given a point $x^* \in \XX$, let the function $\map{V}{\XX}{\realnonnegative}$ satisfy 
	\begin{align}
		V(x^*) = 0, \quad V(x) > 0 \quad \forall x \in \XX \setminus \{x^*\}.
	\end{align}
	Assume there exists a continuous function $\map{\phi}{\XX}{\realnonnegative}$ such that $\phi(x^*) = 0$, $\phi(x) > 0$ for all $x \in \XX\setminus \{x^*\}$, and 
	\begin{align}\label{eq:V-lyap}
		V(f(x))-V(x)\leq-\phi(x) \quad \text{ for all } x \in \XX.
	\end{align}
	Then, any trajectory $\{x_t\}$ of~\eqref{eq:gen-sys} satisfies $\lim_{t \to \infty} x_t = x^*$. 
\end{proposition}
\begin{proof}
	By contradiction, assume that there exists a trajectory $\{x_t\}_{t=0}^\infty$ of~\eqref{eq:gen-sys} such that $\lim_{t \to \infty} x_t \not = x^*$. Using this fact and the compactness of $\XX$, there exists a subsequence of $\{x_t\}_{t=0}^\infty$, denoted as $\{x_{t_k}\}_{k=0}^\infty$, such that $\lim_{k\rightarrow\infty}x_{t_k}=\bar{x}$ and $\bar{x}\neq x^*$. 
Let $\epsilon>0$ be such that $\NN_\epsilon:=\setdef{x \in \XX}{\norm{x-\bar{x}}\leq\epsilon}$ does not contain $x^*$. Let $\bar{\phi}:=\min\limits_{x\in\NN_\epsilon}\phi(x) > 0$. This is well defined as $\NN_\epsilon$ is compact and $\phi$ is continuous. Since $x_{t_k} \to \bar{x}$, there exists a $K$ such that $x_{t_k} \in \NN_\eps$ for all $k \ge K$. Using~\eqref{eq:V-lyap} and the definition of $\bar{\phi}$,  we have
\begin{align}
	V(x_{t_{k+1}})  \le V(x_{t_k}) - \phi(x_{t_k} ) 
 \le V(x_{t_k}) - \bar{\phi} \label{eq:lyap-inf-dec},
\end{align}
for all $k \ge K$. The sequence $\{V(x_t)\}_{t=1}^\infty$ is non-increasing due to~\eqref{eq:V-lyap}. This fact along with~\eqref{eq:lyap-inf-dec} and the lower bound on $V$ yields a contradiction. This completes the proof. 
\end{proof}
\end{document}